\documentclass{article}
\usepackage{amssymb}
\usepackage{amsfonts}
\usepackage{amsmath}
\usepackage{mathrsfs}
\usepackage{mathtools}
\usepackage{graphicx}
\usepackage{tikz}
\usetikzlibrary{trees} 
\newtheorem{theorem}{Theorem}

\newtheorem{definition}[theorem]{Definition}

\newtheorem{proposition}[theorem]{Proposition}
\newtheorem{remark}[theorem]{Remark}

\newenvironment{proof}[1][Proof]{\noindent \textbf{#1.} }{\  \rule{0.5em}{0.5em}}

\begin{document}
\title{ \Large\bf A short construction of the Lie algebra $G_{2}(K)$ over fields $K$ of characteristic $2$}
\author{   \bf  Mashhour  Bani-Ata\footnote{The correspondence author}\;\;, Abdulkareem Alhuraiji }
\date{Department of Mathematics, College of Basic Education \\ Public Authority for Applied Education and Training \\Ardiyah -Kuwait\\  mashhour\_ibrahim@yahoo.com \\
abdulkareemalhuriaji@gmail.com}

\maketitle

\begin{abstract}
The purpose of this paper is to give an explicit and elementary construction for the Lie algebras of type $G_{2}(K)$ of dimension $14$, over the field $K$ of characteristic $2$. We say an elementary construction on the account that we use not more than little naive linear algebra notions.
 
\end{abstract}
\begin{flushleft}


\end{flushleft}

\section{Introduction} 
We use the generalized quadrangle $(\mathbb{P}, \mathcal{L})$ of type $O_{6}^-(2)$, where $\mathbb{P}$ is a set of points and $\mathcal{L}$ is a set of lines, to introduce the notion of Lie-roots $R_\Delta$, $\Delta$ is an element in the set of root bases $\Phi$ and a $27$-dimensional module $A= \langle e_x | x \in \mathbb{P} \rangle$, to prove the following results.
\begin{enumerate}
    \item[*] If $L\in \mathcal{L}$, $W$ is the Weyl group of type $E_{6}$ and $d\in N_{W}(L)$ of order $3$. Then we show that the centralizer $C_{D_L}(d)$ of $d$ in the Lie algebra $D_L$ of type $D_{4}$ is a Lie algebra of type $G_{2}(K)$ of dimension $14$.
    \item[*] We show that $G_{2}(K)$ is an embedded Lie subalgebra of a $28$-dimensional Lie algebra of type $D_{4}$.
    \item[*] For root bases $\Delta$, $X$ in $\Phi$, we show that \begin{align*}
        G_2(K)=&C_{H_L}(d)\oplus \langle R_\Delta , R_X | \Delta , X \in \Phi , R_\Delta^d=R_\Delta,\\
        &S_X=R_X +R_X^d +R_X^{d^2}, R_X^d \neq R_X \rangle,
    \end{align*}
    where $C_{H_L}(d)$ is the centralizer of $d$ in  $H_L$, and $H_L$ is a Cartan Lie subalgebra of $D_4(K)$, of dimension $4$.
 \end{enumerate}
	Here we follow a geometric approach to introduce the notion of root bases and Lie roots, to give a short and explicit construction of $G_2(K)$ and its embedding in the Lie algebra of type $D_4$ over fields $K$ of characteristic $2$, which is a notorious case. This could lead to much understanding of this type of Lie algebras.
There are five exceptional Lie algebras of type $G_2$, $F_4$, $E_6$, $E_7$ and $E_8$ of dimension $14$, $56$, $78$, $133$ and $248$ respectively. The construction of these expcetional Lie algebras has been treated by several authors begining with Cartan \cite{2}, W. Killing \cite{9}, Tits \cite{13}, Harish-Chandra \cite{5}, Jacobson \cite{7,8}, Adam \cite{1}, Fulton and Haris \cite{4}. In \cite{14}, N. Wildberger gave a construction of $G_2$ over the complex numbers requiring no knowledge of Lie theory. The direct sum decomposition of Lie algebras of type $G_2$ has been discussed by Zhu, Hao,  Xin in \cite{15}. The construction of the embedding $D_4< F_4 <E_6$ over fields of characteristic $2$, has been given in \cite{16}. The construction of Lie algebras of type $E_6(K)$ for fields $K$ of characteristic $2$, has been given in \cite{18}. For more information about $G_2$ one may refer to \cite{17} , \cite{19} ,\cite{k} and \cite{h}.

\section{Preliminaries and Notation}
We start with a vector space $V$ over $\mathbb{F}_2$ of dimension $6$ and a non-degenerate quadratic form $Q$ on $V$, of minimal Witt-index . That is, the maximal dimension of a totally singular subspace is $2$. As all quadratic spaces of minimal Witt-index are isomorphic, we can consider $V=\mathbb{F}_4^3$  as vector space over $\mathbb{F}_2$ and \begin{equation*}
    Q(x_1,x_2,x_3)=x_1\overline x_1 + x_2\overline x_2 +x_3\overline x_3 \text{ where } \overline x =x^2.
 \end{equation*}
 Let $\mathbb{P}$ be the quadric of $Q$ i.e $\mathbb{P}= \{0\neq x \in V| Q(x)=0\}$, and $(.|.)$ the corresponding bilinear form \begin{equation*}
   (x|y)=Q(x+y)-Q(x)-Q(y).  
 \end{equation*}
 The order of $\mathbb{P}$ is $27$, the elements of $\mathbb{P}$ are called points and those vectors $0 \neq x \in V$ with $Q(x)\neq 0$ are called esterior points, and hence there are $36$ exterior points. Let \begin{equation*}
     \mathcal{L} = \{ X \leq V| dim X = 2 \text{ and }Q(x)=0 \text{ for all }x\in X \}.
 \end{equation*} The order $|\mathcal{L}|$ of $\mathcal{L}$ is $45$. The elements of $\mathcal{L}$ are called lines.

Each $x\in \mathbb{P}$ is contained in exactly $5$ lines of $\mathcal{L}$ and the quandrangle property also holds, that is if   $L\in \mathcal{L}$, $x\in \mathbb{P}$ with $x\notin L$, then there exists a unique $y\in L$ such that $x,y$ are contained in a line. The pair $(\mathbb{P}, \mathcal{L})$ is a generalized quadrangle of type $O_6^-(2)$. Define \begin{equation*}
    W = \{ g\in GL(V)| Q(x^g)=Q(x) \text{ for all } x\in V \}.
\end{equation*} This group is of type $\Omega_6^-(2).2=U_4(2):2$, and it is a $3$-transposition group generated by its $36$ reflections, i.e by the transformations $\sigma_v$ for \begin{equation*}
    v\in V\text{, }Q(v)=0 \text{ where } x^{\sigma_v}=x+(x|v)v \text{ for all } x\in V.
\end{equation*} All reflections $\sigma_v$, $Q(v)=0$, generate an automorphism group $O(V,Q)$ of $(V, Q)$, also denoted as the corresponding orthogonal group of \begin{equation*}
    O_6^-(2)=\Omega_6^-(2).2\cong U_4(2).2.
\end{equation*}
These reflections $\sigma_v$ form a class of $3$-transpositions in $W$ and generate $W$. The group $W$ is the Weyl group of type $E_6$ and of order $51840$. For the above observation one may refer to \cite{16} and \cite{18}.

\begin{remark} \label{re6}
The Weyl group is transitive on $\mathbb{P}$ by Witt's Lemma, that is for an orthogonal space $V$ with $U_1,U_2\leq V$ are two subspaces of $V$, such that $(U_1,Q|_{U_1})$ and $(U_2,Q|_{U_2})$, where $Q|_{U_{i}}$, $i=1,2$, is the restriction of $Q$ on $U_{i}$, are isomorphic i.e there exists a linear transformation $h: U_1 \rightarrow U_2$ such that $Q(u_1) = Q(u_1^h)$ for all $u_1\in U_1$, then $h$ can be extended to an isomorphism of $V$ i.e there exists $g\in O(V,Q)$, with $U_1^g=U_2$ and $g|_{U_1}=h$.    
\end{remark}

\begin{definition} \label{de7}
     A subset $\Delta$ of $\mathbb{P}$ is called a root base if $\Delta$ is a $\mathbb{F}_2$-base of $V$ and for any two distinct elements. $x,y\in \Delta$, $(x|y)=1$. The set of root bases in $\mathbb{P}$ is denoted by $\Phi$.
\end{definition}

\begin{remark} \label{r7}
     If $\Delta$ is a root base, then $s_\Delta =\displaystyle\sum_{x\in \Delta} x$ is an exterior point, and $\Delta^*=\Delta+s_\Delta$ is also a root base. We call $\Delta$ and $\Delta^*$ corresponding root bases. \begin{equation*}
         \text{Let } \Delta_0=\{ v\in \mathbb{P} | (v|s_\Delta)=0\}, \text{ then } \mathbb{P}= \Delta_0\cup \Delta \cup \Delta^*.
     \end{equation*} 
     A line is either contained in $\Delta_0$ or it is a transversal line i.e it intersects $\Delta$, $\Delta^*$ and $\Delta_0$.
\end{remark}

\begin{proposition} \label{p8}
    \begin{enumerate}
        \item[(1)] $\mathbb{P}$ contains exactly $72$ root bases, i.e $|\Phi|=72$.
        \item[(2)] The group $W$ acts transitivity on $\Phi$ with stabilizer $S_6$ and imprimitively on $\Phi$ with $36$ blocks $\{ \Delta, \Delta^* \}$, $\Delta\in \Phi$.
    \end{enumerate}
\end{proposition}
\begin{proof}
	See \cite{18}.
\end{proof}

\begin{remark} \label{r9}
Root bases can be constructed as follows. For two points $x$ and $y$ with $( x |y)=1$, set \begin{equation*}
\Delta_{x,y}=\{x\} \cup \{ v\in \mathbb{P} | (x|v)=1, (y|v)=0 \}    
\end{equation*}
For a root base $\Delta$ and $x \in \Delta$ one has a root base $\Delta_{x,x+s_\Delta}$. It is obvious that for a point $x$, there exist $16$ points $y$ with $( x |y)=1$. Hence again $|\Phi|=\frac{27.16}{6}=72$. Any root base corresponds to an anisotropic vector $s_{\Delta} =\displaystyle\sum_{b\in \Delta} b$ and every anisotropic vector $s$ or a reflection $\sigma_{s}$ corresponds to two root bases $\Delta_{1} = \Delta_{2}$ such that $s = s_{\Delta}$. Moreover,\begin{equation*}
    \Delta_2=\Delta_1^{\sigma_s}=\Delta_1+s \text{ and } \Delta_1\cup \Delta_2=\{ v\in \mathbb{P}| (s,v)=1 \}.
\end{equation*}  We denote $\sigma_{s_\Delta}$ by $\sigma_{\Delta}$ and $\Delta^{\sigma_s}$ by $\Delta^{*}$.
\end{remark}

\begin{proposition} \label{p10} 
For root bases $\Delta$ and $\Gamma$ with $s_\Delta\neq s_\Gamma$ it holds\begin{enumerate}
    \item[1.]lf $(s_\Delta |s_\Gamma )=0$, then $| \Delta \cap \Gamma|=|\Delta \cap \Gamma^* |=1$ and $\Delta ^ {\sigma_{\Gamma}} = \Delta$.
    \item[2.] If $(s_\Delta|s_\Gamma)=1$, then
    \item[(i)] $| \Delta \cap \Gamma|=|\Delta^* \cap \Gamma^* |=3$ and $ \Delta \cap \Gamma^*=\Delta^* \cap \Gamma=\phi$ or 
    \item[(ii)] $| \Delta \cap \Gamma^*|=|\Delta^* \cap \Gamma |=3$ and $ \Delta \cap \Gamma=\Delta^* \cap \Gamma^*=\phi$.
    \item[(iii)] If $\Delta \cap \Gamma=\phi$ then $(\Delta^* \cap \Gamma)^{\sigma_\Delta} \cup (\Gamma^* \cap \Delta)^{\sigma_\Gamma}= \Delta^{\sigma_\Gamma}= \Gamma^{\sigma_\Delta} $ is a root base corresponding to $s_\Delta +s_\Gamma$. 
\end{enumerate}
\end{proposition}
\begin{proof}
	See \cite{18}
\end{proof}
\section{Lie algebras of type $E_{6}$ and of characteristic $2$}

Let $K$ be a field of characteristic $2$ and $(\mathbb{P}, \mathcal{L})$ a generalized quadrangle of type $O_6^- (2)$ as defined above, and let $A$ be a $27$-dimensional vector space over $K$ with base $\{ e_x | x \in \mathbb{P} \}$. Let $End_K(A)$ denote the Lie algebra consisting of all elements of $GL(A)$ with Lie product $[X, Y] = XY - YX$. For $v \in V$ define $H_{v} \in End(A)$ by $e_{x} ^ {H_{v} }=(x|v)e_x$, $x\in \mathbb{P}$ and for a root base $\Delta$ define the Lie root $R_{\Delta}$ on $A$ by \begin{equation*}
e_{x} ^ {R_{\Delta} } = \left\{ \begin{array}{cl}
e_{x} ^ {\sigma_{\Delta} }=e_{x+s_\Delta} & , \ x \in \Delta \\
0 & , \ \text{otherwise}
\end{array} \right.    
\end{equation*}

It is obvious that $R_{\Delta}$ has rank $6$ and $R_\Delta$, $R_{\Delta^*}$ are transposed to each other with respect to the base $\{ e_{x}|x \in \mathbb{P}\}$ and $R_{\Delta} ^ 2 = 0$.

\begin{proposition} \label{p11}
For root bases $\Delta$ and $\Gamma$ and for $v,w\in V$ it holds:
\begin{enumerate}
    \item[1.] $[H_v, H_w]=0$ and $[H_v, R_\Delta]=(s_\Delta|v)R_\Delta$.
    \item[2.] $[R_\Delta,R_{\Delta^*}]=H_{s_\Delta}$.
    \item[3.] If $s_\Delta$ and $s_\Gamma$ are distinct and orthogonal, then $[R_\Delta, R_\Gamma]=0$.
    \item[4.] If $(s_\Delta | s_\Gamma)=1$ and $\Delta \cap \Gamma\neq \phi $, then $[R_\Delta, R_\Gamma]=0$.
    \item[5.] If $(s_\Delta | s_\Gamma)=1$ and $\Delta \cap \Gamma= \phi $, then $[R_\Delta, R_\Gamma]= R_{\Gamma^{\sigma_\Delta}}= R_{\Delta^{\sigma_\Gamma}}$.
    \item[6.] If $\Delta\neq \Gamma$, then $R_\Delta R_\Gamma R_\Delta=R_\Delta[R_\Gamma , R_\Delta]=0$. 
\end{enumerate}
\end{proposition}
\begin{proof}
For the proof see \cite{18} and \cite{20}.
\end{proof}

\begin{remark} \label{d12}
    The Lie algebra $\mathbb{E}$ generated by the elements $H_{v}, v \in V$ and the $72$ Lie roots $R_\Delta$, $\Delta \in \Phi$ is a subalgebra of $End_K(A)$ of dimension $78$ over $K$. Proposition \ref{p11} implies that the Lie algebra $\mathbb{E}$ is of type $E_6$. The group $W$ acts on $A$ by $e_{x} ^ g = e_{x^g}$, $W$ permutes the roots $R_\Delta$ and induces a Weyl group on the Lie algebra $\mathbb{E}$, normalizing the Cartan subalgebra $\mathbb{H} = \langle H_v | v \in V \rangle$.
\end{remark}.

\begin{theorem} \label{t13}
     Let $L$ be a line and $\Phi_L =\{ \Delta \in \Phi | s_\Delta \in L^\perp\}$. Then \begin{equation*}
         D_L=\langle R_\Delta|\Delta \in \Phi _L \rangle
     \end{equation*}  is a Lie subalgebra of $F_4$ of dimension $28$.
\end{theorem}
\begin{proof}
	See \cite{16}.
\end{proof}

\section{The Lie algebra $G_2(K)$}
 \begin{proposition} \label{p1}
     Let $\Delta$ be a root base, $L\in \mathcal{L}$. Then \begin{enumerate}
         \item[(1)] The set $R_L=\{R_\Delta| L\subseteq s_\Delta^\perp \}$ has size $24$.
         \item[(2)] $H_L= \langle H_v | L \subseteq v^\perp ,v\in V\rangle$ is of dimension $4$.
         \item[(3)] $D_L=\langle H_L,R_L \rangle =H_L \underset{R_\Delta\in R_L}{\oplus} R_L$
         \item[(4)] $D_L$ is closed under Lie-commutator multiplication and it is a Lie algebra of type $D_4(K)$ of dimension $24+4=28.$
         \item[(5)] $D_4(K)$ is embedded in the Lie algebra $F_4 (K)$ of dimension $56$.
     \end{enumerate}
 \end{proposition}
 \begin{proof}
 	See \cite{16}.
 \end{proof}

Now we have the following main theorem.
\begin{theorem} \label{t2}
Let $L\in \mathcal{L}$, $d\in N_W (L)$ and $order (d)=3$, where $W$ is the Weyl group of type $E_6$. Then the centralizer $C_{D_L}(d)$ of $d$ in $D_L$ is a Lie subalgebra of type $G_2$ and of dimension $14$.    
\end{theorem}

To prove Theorem \ref{t2}, the following remark and propositions are needed.
\begin{remark} \label{r3}
The Weyl group $W$ permutes the points of $\mathbb{P}$ induces permutations on $27$ elements. So $W$ acts on $A$ by permutation, if $w\in W$, then $e_x^w=e_{x^w}$ where $A=\langle e_x|x\in \mathbb{P} \rangle$ and if $\Delta$ is a root base and $R_\Delta$ is a Lie root in $End_K (A)$, then $R_\Delta^w=P_w^{-1} R_\Delta P_w =R_{\Delta^w}$ where $P_w$ is the permutation induced by $w$ on $\mathbb{P}$. Hence if $d\in N_W(L)$ of order $3$, then $R_\Delta^d=R_{\Delta^d}$. Set $R_\Delta +R_\Delta^d + R_\Delta ^{d^2}= S_\Delta$, then it follows that $S_\Delta^{d}=S_\Delta^{d^2}=S_\Delta$ and this leads to the following proposition     
\end{remark}

\begin{proposition} \label{p4}
 Let $d\in N_W(L)$, $L\in \mathcal{L}$ and $order(d)=3$. Then $d$ has $6$ fixed Lie roots in $R_L$ and $6$ orbits of length $3$.   
\end{proposition}
 \textbf{Proof. } As $d$ permutes the Lie roots in $R_L$, then we have $6$ fixed Lie roots \begin{equation*}
     S_{\Delta_i}=R_{\Delta_i}+R_{\Delta_i}^d +R_{\Delta_i}^{d^2}\text{, }R_{\Delta_i}^d\neq R_{\Delta_i},
 \end{equation*} $i=1,...,6$, and $6$ orbits of length $3$ of the form 
 \[
\begin{tikzpicture}[baseline={(0,-0.3)}]
  \node (A) at (0.5,0) {$\Delta$};
  \node[circle, fill=black, inner sep=1pt] (B) at (1,0) {};
  \node[circle, fill=black, inner sep=1pt] (k) at (1.7,0) {};
  \node (C) at (2,0) {$\Delta^{d}$};
  \node [circle, fill=black, inner sep=1pt](D) at (2.35,0) {};
  \node[circle, fill=black, inner sep=1pt] (J) at (3.1,0) {};
  \node (E) at (3.5,0) {$\Delta^{d^2}.$};

  \draw (B) -- (C);
  \draw (D) -- (J);
\end{tikzpicture}
\]
This is equivalent to $C_{D_L}(d)$ has Lie roots $\{ R_\Delta | R_\Delta^d = R_\Delta \}$ and Lie roots \begin{equation*}
     \{ R_X+ R_X^d + R_X^{d^2} | X \text{ is a root base, } R_X^d \neq R_X \}.
 \end{equation*} 
 The total number of these Lie roots is $6.3+6=24$.

\begin{proposition} \label{p5}
Let $X$ be a root base. Define the Lie roots  \begin{equation*}
S_X = \left\{ \begin{array}{cl}
R_X & \text{if} \ R_X^d=R_X \\
R_X+ R_X^d + R_X^{d^2}  & \text{if} \ R_X^d \neq R_X \;\; .
\end{array} \right.    
\end{equation*}
 
 \text{Then} $G=C_{H_L}(d) \oplus \langle S_X \rangle$ is a Lie algebra of type $G_2< D_4$, where $C_{H_L}(d)$ is a Cartan Lie subalgebra of dimension $2$.    
\end{proposition}
\begin{proof}

The order of the set $\{S_X\}$ is $6+6=12$. Now we have to show that $G$ is closed under the Lie multiplication, i.e $[S_X,S_Y]$ is an element of $C_{H_L}(d)$, or $0$ or a Lie root $S_Z$ for some root base $Z$. So we  compute the following cases
\textbf{Case(1). }Let $R_\Delta$ be in $\{S_X\}$ such that $R_\Delta^d=R_\Delta$ and $R_\Gamma+R_\Gamma^d +R_\Gamma^{d^2}$ in $\{ S_X \}$ where $R_\Gamma^d \neq R_\Gamma$, then \begin{align*}
        [R_\Delta, R_\Gamma+R_\Gamma^d +R_\Gamma^{d^2}] &=[R_\Delta, R_\Gamma]+ [R_\Delta, R_\Gamma^d] + [R_\Delta, R_\Gamma^{d^2}]\\
        &=[R_\Delta, R_\Gamma]+ [R_\Delta^d, R_\Gamma^d] + [R_\Delta^{d^2}, R_\Gamma^{d^2}]\\
        &=[R_\Delta, R_\Gamma]+ [R_\Delta, R_\Gamma]^d + [R_\Delta, R_\Gamma]^{d^2}\\
        &=0 \text{ or } R_D+R_D^d+R_D^{d^2}
    \end{align*}
    for some root base $D$ by Proposition \ref{p11}.\\
\textbf{Case(2). }We compute $[R_\Delta , R{_\Delta^*}]$. Let $R_\Delta=M$, $R_{\Delta^*}=N$, $x,y\in \mathbb{P}$, then $(MN)_{xy}=\displaystyle\sum_{z}M_{xz} N_{zy} = 1$ if and only if $x=y\in \Delta\cap \Delta^{*^{\sigma_\Delta}}=\Delta$ and $(NM)_{uv}=1$ if and only if $u=v\in \Delta^* \cap \Delta^{\sigma_\Delta}=\Delta^*$. Hence $[M,N]$ is diagonal having $1's$ exactly in positions $(x,x)$ for $x\in \Delta\cup \Delta^*= \{ x| (s_\Delta |x)=1 \}$. Hence $[R_\Delta, R_{\Delta^*}]=H_{s_\Delta}$ where $s_\Delta \in L^\perp \setminus L$ and $H_{s_\Delta} \in C_{H_L}(d)$.\\ 
    \textbf{Case(3). }We consider the case $[R_\Delta+ R_\Delta^d+ R_\Delta^{d^2},R_\Gamma+ R_\Gamma^d+ R_\Gamma^{d^2}]$, where $R_\Delta^d\neq R_\Delta$ and $R_\Gamma^d\neq R_\Gamma$. So one has 
    \begin{align*}
        &[R_\Delta+ R_\Delta^d+ R_\Delta^{d^2},R_\Gamma+ R_\Gamma^d+ R_\Gamma^{d^2}]=\\
        &[R_\Delta, R_\Gamma] + [R_\Delta, R_\Gamma^d]+ [R_\Delta, R_\Gamma^{d^2}]+
        [R_\Delta^d, R_\Gamma] + [R_\Delta^d, R_\Gamma^d]+ [R_\Delta^d, R_\Gamma^{d^2}]+\\
        &[R_\Delta^{d^2}, R_\Gamma] + [R_\Delta^{d^2}, R_\Gamma^d]+ [R_\Delta^{d^2}, R_\Gamma^{d^2}]=\\
        &[R_\Delta, R_\Gamma] +[R_\Delta, R_\Gamma]^d +[R_\Delta, R_\Gamma]^{d^2} +[R_\Delta, R_\Gamma^d] + [R_\Delta^d, R_\Gamma^{d^2}] + [R_\Delta^{d^2}, R_\Gamma]+\\
        &  [R_\Delta, R_\Gamma^{d^2}]+[R_\Delta^d, R_\Gamma]+ [R_\Delta^{d^2}, R_\Gamma]=\\
        &[R_\Delta, R_\Gamma] +[R_\Delta, R_\Gamma]^d +[R_\Delta, R_\Gamma]^{d^2} +[R_\Delta, R_\Gamma^d] +[R_\Delta, R_\Gamma^{d}]^d + [R_\Delta, R_\Gamma^d]^{d^2}+\\
        & [R_\Delta, R_\Gamma^{d^2}]+ [R_\Delta, R_\Gamma^{d^2}]^d+ [R_\Delta, R_\Gamma^{d^2}]^{d^2}\\
        &= (R_{D_0} + R_{D_0}^d + R_{D_0}^{d^2})+(R_{D_1}+ R_{D_1}^d + R_{D_1}^{d^2})+(R_{D_2}+ R_{D_2}^d + R_{D_2}^{d^2})
    \end{align*}
    where \begin{equation*}
        [R_\Delta, R_\Gamma] = \left\{ \begin{array}{cl}
0 & \\
\text{or } R_{D_0} & \\
\text{or in } C_{H_L}(d) & 
\end{array} \right. ,
[R_\Delta, R_\Gamma^d] = \left\{ \begin{array}{cl}
0 & \\
\text{or }R_{D_1} & \\
\text{or in } C_{H_L}(d) & 

\end{array} \right.
    \end{equation*}
and \begin{equation*}
    [R_\Delta, R_\Gamma^{d^2}] = \left\{ \begin{array}{cl}
0 & \\
\text{or }R_{D_2} & \\
\text{or in } C_{H_L}(d) & 

\end{array} \right.
\end{equation*} by Proposition \ref{p11}. This completes the proof.
\end{proof}\\
\textbf{Proof of Theorem \ref{t2}. } It is an immediate consequence of Propositions \ref{p5}, \ref{p4} and hence the centralizer $C_{D_L}(d)$ is a Lie algebra of type $G_2$. It has dimension $14$ as $dim\langle S_X\rangle =12$ and $dim C_{H_L}(d)=2$, moreover it is a Lie subalgebra of $D_4$. Hence the claim.
\section{Acknowledgment}
The authors would like to  thank the Public Authority for Applied Education and Training for supporting this research project No BE-25-13.\\

\end{document}